\documentclass{article}%
\usepackage{amsmath}
\usepackage{amsfonts}
\usepackage{amssymb}
\usepackage{graphicx}%
\providecommand{\U}[1]{\protect\rule{.1in}{.1in}}
\newtheorem{theorem}{Theorem}
\newtheorem{acknowledgement}[theorem]{Acknowledgement}

\newtheorem{corollary}[theorem]{Corollary}

\newtheorem{lemma}[theorem]{Lemma}

\newtheorem{remark}[theorem]{Remark}

\newenvironment{proof}[1][Proof]{\noindent\textbf{#1.} }{\ \rule{0.5em}{0.5em}}
\begin{document}

\title{Free Products of Generalized RFD C*-algebras}
\author{Don Hadwin}
\maketitle

\begin{abstract}
If $k$ is an infinite cardinal, we say a C*-algebra $\mathcal{A}$ is
residually less than $k$ dimensional, $R_{<k}D,$ if the family of
representations of $\mathcal{A}$ on Hilbert spaces of dimension less than $k$
separates the points of $\mathcal{A}.$ We give characterizations of this
property, and we show that if $\left\{  \mathcal{A}_{i}:i\in I\right\}  $ is a
family of $R_{<k}D$ algebras, then the free product $\underset{i\in I}{\ast
}\mathcal{A}_{i}$ is $R_{<k}D$. If each $\mathcal{A}_{i}$ is unital, we give
sufficient conditions, depending on the cardinal $k$, for the free product
$\underset{i\in I}{\ast_{\mathbb{C}}}\mathcal{A}_{i}$ in the category of
unital C*-algebras to be $R_{<k}D$. We also give a new characterization of
RFD, in terms of a lifting property, for separable C*-algebras.

\end{abstract}

\bigskip

\bigskip

\section{Introduction}

A C*-algebra $\mathcal{A}$ is \emph{residually finite dimensional }( $RFD$ )
if the collection of all finite-dimensional representations of $\mathcal{A}$
separate the points of $\mathcal{A}$; equivalently, if there is a direct sum
of finite-dimensional representations of $\mathcal{A}$ with zero kernel. It is
clear that every commutative C*-algebra is RFD. Man-Duen Choi \cite{C} showed
that free group C*-algebras are RFD. Ruy Exel and Terry Loring \cite{EL}
proved that the free product of two RFD algebras is RFD. The class of RFD
C*-algebras plays an important role in the theory of C*-algebras, e.g.,
\cite{A}, \cite{B}, \cite{BO}, \cite{C}, \cite{D}, \cite{EL}, \cite{GM},
\cite{L}, \cite{HLLS}.

In this paper we introduce a related notion. Suppose $k$ is an infinite
cardinal. We say that a C*-algebra $\mathcal{A}$ is \emph{residually less than
}$k$\emph{-dimensional}, conveniently denoted by $R_{<k}D$, if the class of
representations of $\mathcal{A}$ on Hilbert spaces of dimension less than $k$
separates the points of $\mathcal{A}$; equivalently, if there is a direct sum
of such representations that has zero kernel. Note that when $k=\aleph_{0}$,
we have $R_{<k}D$ is the same as $RFD.$ We give characterizations of $R_{<k}D$
algebras that show that the free product of an arbitrary collection of
$R_{<k}D$ C*-algebras is $R_{<k}D$. We also give conditions that ensure that
the free product (amalgamated over $\mathbb{C}$) of unital C*-algebras in the
category of unital C*-algebras is $R_{<k}D$; this always happens when each of
the algebras has a one-dimensional unital representation.

The proofs rely on a simple result (Lemma \ref{simple}) and results of the
author \cite{H1}, \cite{H2} on approximate unitary equivalence and approximate
summands of nonseparable representations of nonseparable C*-algebras.

Suppose $k$ and $m$ are infinite cardinals. We say that a C*-algebra
$\mathcal{A}$ is $\emph{m}$\emph{-generated} if is generated by a set with
cardinality at most $m$. For each cardinal $s,$ we let $H_{s}$ be a Hilbert
space whose dimension is $s$. If $\pi:\mathcal{A}\rightarrow B\left(
H\right)  $ is a $\ast$-homomorphism, we say that the \emph{dimension} of
$\pi$ is $\dim\pi=\dim H$. We define \textrm{Rep}$_{k}\left(  \mathcal{A}%
\right)  $ to be the set of all representations $\pi:\mathcal{A}\rightarrow
B\left(  H_{s}\right)  $ for some $s<k$.

If $\mathcal{A}$ is a C*-algebra, then $\mathcal{A}^{+}$ denotes the
C*-algebra obtained by adding a unit to $\mathcal{A}$ (that is different from
the unit in $\mathcal{A}$ if $\mathcal{A}$ is unital).

We end this section with our key lemma. Suppose $H$ is a Hilbert space and $P$
is a projection in $B\left(  H\right)  $. We define $\mathcal{M}_{P}=PB\left(
H\right)  P$. Then $\mathcal{M}_{P}$ is a unital C*-algebra, but the unit is
$P$, not $1$. However, $\mathcal{M}_{P}$ is a C*-subalgebra of $B\left(
H\right)  $. A unitary element of $\mathcal{M}_{P}$ is an operator $U\in
B\left(  H\right)  $ such that $UU^{\ast}=U^{\ast}U=P$, and is the direct sum
of a unitary operator on $P\left(  H\right)  $ with $0$ on $P\left(  H\right)
^{\perp}$. If $P\neq1$, a unitary operator in $\mathcal{M}_{P}$ is never
unitary in $B\left(  H\right)  $.

We use the symbol $\ast$-SOT to denote the $\ast$-strong operator topology.

\bigskip

\begin{lemma}
\label{simple}Suppose $\left\{  P_{\alpha}\right\}  $ is a net of projections
in $B\left(  H\right)  $ such that $P_{\alpha}\rightarrow1$ ( $\ast$-SOT ) and
let
\[
\mathcal{B}=\left\{  \left\{  T_{\alpha}\right\}  \in%
{\displaystyle\prod_{\alpha}}
\mathcal{M}_{P_{\alpha}}:\exists T\in B\left(  H\right)  ,T_{\alpha
}\rightarrow T\text{ ( }\ast\text{-SOT )}\right\}  ,
\]
and
\[
\mathcal{J}=\left\{  \left\{  T_{\alpha}\right\}  \in\mathcal{B}:T_{\alpha
}\rightarrow0\text{ ( }\ast\text{-SOT )}\right\}  ,
\]
and define $\pi:\mathcal{B}\rightarrow B\left(  H\right)  $ by
\[
\pi\left(  \left\{  T_{\alpha}\right\}  \right)  =\text{(}\ast\text{-SOT)-}%
\lim_{\alpha}T_{\alpha}.
\]
Then

\begin{enumerate}
\item $\mathcal{B}$ is a unital C*-algebra,

\item $\mathcal{J}$ is a closed two-sided ideal in $\mathcal{B}$,

\item If $T\in\mathcal{B}$, then $\pi\left(  \left\{  P_{\alpha}TP_{\alpha
}\right\}  \right)  =T,$

\item $\pi$ is a unital surjective $\ast$-homomorphism

\item If $U\in B\left(  H\right)  $ is unitary, then there is a unitary
$\left\{  U_{\alpha}\right\}  \in\mathcal{B}$ such that
\[
\pi\left(  \left\{  U_{\alpha}\right\}  \right)  =U.
\]

\end{enumerate}
\end{lemma}

\begin{proof}
Statements $\left(  1\right)  $-$\left(  4\right)  $ are easily proved. To
prove $\left(  5\right)  $, note that if $U\in B\left(  H\right)  $ is
unitary, then there is an $A=A^{\ast}\in B\left(  H\right)  $ such that
$U=e^{iA}$. We can easily choose $A_{\alpha}=A_{\alpha}^{\ast}$ for each
$\alpha$ so that $\pi\left(  \left\{  A_{\alpha}\right\}  \right)  =A.$ Thus,
if $U_{\alpha}=e^{iA_{\alpha}}$ (in $\mathcal{M}_{P_{\alpha}})$, then
$\left\{  U_{\alpha}\right\}  $ is unitary in $\mathcal{B}$ and $\pi\left(
\left\{  U_{\alpha}\right\}  \right)  =U$.
\end{proof}

\bigskip

Here is a simple application that gives the flavor of our results.

\bigskip

\begin{corollary}
Every free group is RFD.
\end{corollary}

\begin{proof}
Suppose $\mathbb{F}$ is a free group and $\mathcal{A=C}^{\ast}\left(
\mathbb{F}\right)  =C^{\ast}\left(  \left\{  U_{g}:g\in\mathbb{F}\right\}
\right)  .$ Choose a Hilbert space $H$ and a faithful representation
$\rho:\mathcal{A}\rightarrow B\left(  H\right)  $. Choose a net $\left\{
P_{\alpha}\right\}  $ of finite-rank projections such that $P_{\alpha
}\rightarrow1$ ( $\ast$-SOT ). Applying Lemma \ref{simple} we have, for each
$g\in\mathbb{F}$, we can find a unitary element $\left\{  U_{g,\alpha
}\right\}  $ in $\mathcal{B}$ so that $\pi\left(  \left\{  U_{g,\alpha
}\right\}  \right)  =U_{g}$. For each $\alpha,$ we have a unitary group
representation $\sigma_{\alpha}:\mathbb{F}\rightarrow\mathcal{M}_{P_{\alpha}}$
defined by%
\[
\sigma_{\alpha}\left(  g\right)  =U_{g,\alpha}.
\]
By the definition of $C^{\ast}\left(  \mathbb{F}\right)  $, there is a $\ast
$-homomorphism $\tau_{\alpha}:\mathcal{A}\rightarrow\mathcal{M}_{\alpha}$ such
that $\tau_{\alpha}\left(  U_{g}\right)  =U_{g,\alpha}$. It follows that
$\tau:\mathcal{A}\rightarrow\mathcal{B}$ define by $\tau\left(  U_{g}\right)
=\left\{  U_{g,a}\right\}  $ is a $\ast$-homomorphism such that $\pi\circ
\tau=\rho$. Hence the direct sum of the $\tau_{\alpha}$'s is faithful, which
shows that $\mathcal{A}$ is $RFD$.
\end{proof}

\bigskip

The following corollary is from \cite[Exercise 7.1.4]{BO}.

\begin{corollary}
Every C*-algebra is a $\ast$-homomorphic image of an RFD C*-algebra.
\end{corollary}

\begin{proof}
Suppose $\mathcal{A}$ is a C*-algebra. We can assume that $A\subseteq B\left(
H\right)  $ for some Hilbert space $H$. Choose a net $\left\{  P_{\alpha
}\right\}  $ of finite-rank projections converging $\ast$-strongly to $1,$ and
let $\mathcal{B},\mathcal{J}$ and $\pi$ be as in Lemma \ref{simple}. Then
$\mathcal{B}$, and thus $\pi^{-1}\left(  \mathcal{A}\right)  $, is $RFD$ and
$\pi\left(  \pi^{-1}\left(  \mathcal{A}\right)  \right)  =\mathcal{A}$.
\end{proof}

\section{ $R_{<k}D$ Algebras}

\bigskip

We now prove our main results on $R_{<k}D$ C*-algebras. The following two
lemmas contain the key tools.\bigskip

\begin{lemma}
\label{decomp}Suppose $\aleph_{0}\leq k\leq m$, and $\mathcal{A}$ is $R_{<k}D$
and $m$-generated. Then
\end{lemma}

\begin{enumerate}
\item We can write $H_{m}=%
{\displaystyle\sum_{\lambda\in\Lambda}^{\oplus}}
X_{\lambda}$ with $Card\Lambda=m$, and such that, for every $\lambda\in
\Lambda$, $\dim X_{\lambda}<k$ and there is a unital representation
$\pi_{\lambda}:\mathcal{A}^{+}\rightarrow B\left(  X_{\lambda}\right)  $ such
that the representation $\pi:\mathcal{A}^{+}\rightarrow B\left(  H_{m}\right)
$ defined by $\pi=%
{\displaystyle\sum^{\oplus}}
\pi_{\lambda}$ is faithful. Moreover, this can be done so that, for each
$\lambda_{0}\in\Lambda$, we have $Card\left(  \left\{  \lambda\in\Lambda
:\pi_{\lambda}\thickapprox\pi_{\lambda_{0}}\right\}  \right)  =m.$

\item It is possible to choose the decomposition in $\left(  1\right)  $ so
that, for each cardinal $s<k$, there is a $\lambda\in\Lambda$ such that $\dim
X_{\lambda}=s$.
\end{enumerate}

\begin{proof}
Since $\mathcal{A}$ is $R_{<k}D$, there is a direct sum of representations in
\textrm{Rep}$_{k}\left(  \mathcal{A}\right)  $ whose direct sum is faithful.
Suppose $D$ is a generating set for $\mathcal{A}$ and $Card\left(  D\right)
\leq m$. We can replace $D$ by the $\ast$-algebra over $\mathbb{Q}%
+i\mathbb{Q}$ generated by $D$ without making the cardinality exceed $m$. For
each $a\in D$ we can choose a direct sum of countably many summands from our
faithful direct sum that preserves the norm of $a$. Hence, by choosing
$\aleph_{0}Car\left(  D\right)  $ summands, we get a direct sum that is
isometric on $D$ and thus isometric on $\mathcal{A}$. Since $\aleph
_{0}Car\left(  D\right)  \leq m$. we can replace this last direct sum with a
direct sum of $m$ copies of itself and get a direct sum on a Hilbert space
with dimension $m.$ We can replace this Hilbert space with $H_{m}$ and get a
decomposition as in $\left(  1\right)  $. to get $\left(  2\right)  $ note
that, since $\mathcal{A}^{+}$ has a unital one-dimensional representation, we
know that, for every cardinal $s<k$. there is a representation of
$\mathcal{A}^{+}$ of dimension $s$. If we take one such representation for
each $s<k$ and take a direct sum of $m$ copies of all of them, we get a
representation that has has dimension at most $m$, so we add this as a summand
to the representation we constructed satisfying $\left(  1\right)  .$
\end{proof}

\bigskip

\begin{lemma}
\label{*SOTapprox}Suppose $\mathcal{A}$ is a C*-algebra and $k\leq m$ are
infinite cardinals and $D$ is a generating set for $\mathcal{A}$. Suppose we
can write $H_{m}=%
{\displaystyle\sum_{\lambda\in\Lambda}^{\oplus}}
X_{\lambda}$ and $\pi=%
{\displaystyle\sum^{\oplus}}
\pi_{\lambda}$ as in part $\left(  1\right)  $ of Lemma \ref{decomp}. If
$\rho:\mathcal{A}^{+}\rightarrow B\left(  H_{m}\right)  $ is a unital
representation, then, for every $\varepsilon>0,$ every finite subset
$W\subseteq\mathcal{D}$ and every finite subset $E\subseteq H_{m}$, there is a
finite subset $F\subseteq\Lambda,$ such that, for every finite set $G$ with
$F\subseteq G\subseteq\Lambda,$ if $Q_{G}$ is the orthogonal projection onto $%
{\displaystyle\sum_{\lambda\in G}^{\oplus}}
X_{\lambda}$, then there is a unitary $U\in Q_{G}B\left(  H_{m}\right)  Q_{G}$
such that, for every $a\in W$ and $e\in E$, we have%
\[
\left\Vert \left[  \rho\left(  a\right)  -U_{G}^{\ast}\pi\left(  a\right)
U_{G}\right]  e\right\Vert =\left\Vert \left[  \rho\left(  a\right)
-U_{G}^{\ast}\left(
{\displaystyle\sum_{\lambda\in G}}
\pi_{\lambda}\right)  \left(  a\right)  U_{G}\right]  e\right\Vert
<\varepsilon.
\]

\end{lemma}

\begin{proof}
It follows that if $a\in\mathcal{A}$ and $a\neq0$, then $rank\pi\left(
a\right)  =m=rank\left(  \pi\oplus\rho\right)  \left(  a\right)  $. Hence, by
\cite{H1}, $\pi$ is approximately unitarily equivalent to $\pi\oplus\rho$.
However, by \cite{H2}, $\rho$ is a point-$\ast$-SOT limit of representations
unitarily to $\rho$. Hence there is a net $\left\{  U_{\alpha}\right\}  $ of
unitary operators in $B\left(  H_{m}\right)  $ such that, for every
$a\in\mathcal{A}$,%
\[
\text{(}\ast\text{-SOT)}\lim_{\alpha}U_{\alpha}^{\ast}\pi\left(  a\right)
U_{\alpha}=\rho\left(  a\right)  .
\]
However, the net $\left\{  Q_{F}:F\subseteq\Lambda,\text{ }F\text{ is
finite}\right\}  $ is a net of projections converging $\ast$-strongly to $1.$
Hence, by Lemma \ref{simple}, each $U_{\alpha}$ is a $\ast$-SOT limit of
unitaries in the union of $Q_{F}B\left(  H_{m}\right)  Q_{F}$ ($F\subseteq
\Lambda$, $F$ is finite). The result now easily follows.\bigskip
\end{proof}

\bigskip

\bigskip

\begin{theorem}
\label{characterize}Suppose $\aleph_{0}\leq k\leq m$, and $\mathcal{A}$ is
$m$-generated with a generating set $\mathcal{G}$ with $Card\mathcal{G}\leq
m$. The following are equivalent.

\begin{enumerate}
\item $\mathcal{A}$ is $R_{<k}D.$

\item There is a faithful unital $\ast$-homomorphism $\rho:\mathcal{A}%
^{+}\rightarrow B\left(  H_{m}\right)  $ such that, for every $\varepsilon>0$,
every finite subset $E\subseteq H_{m}$ and every finite subset $W\subseteq
\mathcal{G}$, there is a projection $P\in B\left(  H_{m}\right)  $ and a
unital $\ast$-homomorphism $\tau:\mathcal{A}\rightarrow\mathcal{M}%
_{P}=PB\left(  H_{m}\right)  P$ such that, for every $e$ $\in E$ and every
$a\in W$ we have%
\[
\left\Vert \left[  \tau\left(  a\right)  -\rho\left(  a\right)  \right]
e\right\Vert <\varepsilon\text{.}%
\]

\item There is a faithful unital representation $\rho:\mathcal{A}%
^{+}\rightarrow B\left(  H_{m}\right)  $ and a net $\left\{  P_{\alpha
}\right\}  $ of projections in $B\left(  H_{m}\right)  $, each with rank less
than $k$, such that $P_{\alpha}\rightarrow1$ ($\ast$-SOT) and such that, for
each $\alpha,$ there is a representation $\pi_{\alpha}:\mathcal{A}%
\rightarrow\mathcal{M}_{P_{\alpha}}$ such that, for every $a\in\mathcal{A}$,
we have%
\[
\pi_{\alpha}\left(  a\right)  \rightarrow\rho\left(  a\right)  \text{ (}%
\ast\text{-SOT).}%
\]

\item For every unital representation $\rho:\mathcal{A}^{+}\rightarrow
B\left(  H_{m}\right)  $ there is a net $\left\{  P_{\alpha}\right\}  $ of
projections in $B\left(  H_{m}\right)  $, each with rank less than $k$, such
that $P_{\alpha}\rightarrow1$ ($\ast$-SOT) and such that, for each $\alpha,$
there is a representation $\pi_{\alpha}:\mathcal{A}\rightarrow\mathcal{M}%
_{P_{\alpha}}$ such that, for every $a\in\mathcal{A}$, we have%
\[
\pi_{\alpha}\left(  a\right)  \rightarrow\rho\left(  a\right)  \text{ (}%
\ast\text{-SOT).}%
\]

\end{enumerate}
\end{theorem}

\begin{proof}
$\left(  2\right)  \Longrightarrow\left(  1\right)  $ Let $A$ be the set of
triples $\left(  \varepsilon,E,W\right)  $ ordered by $\left(  \geq
,\subseteq,\subseteq\right)  $. If $\alpha=\left(  \varepsilon,E,W\right)  $
let $\tau_{\alpha}:\mathcal{A}\rightarrow P_{\alpha}B\left(  H_{m}\right)
P_{\alpha}$ guaranteed by $\left(  2\right)  $. Since $\mathcal{G}%
=\mathcal{G}^{\ast}$ we have
\[
\text{(}\ast\text{-SOT)}\lim_{\alpha}\tau_{\alpha}\left(  a\right)
=\rho\left(  a\right)
\]
for every $a\in\mathcal{G}$. Since $\rho$ and each $\tau_{\alpha}$ is a $\ast
$-homomorphism, the set of $a\in\mathcal{A}$ for which ($\ast$-SOT)$\lim
_{\alpha}\tau_{\alpha}\left(  a\right)  =\rho\left(  a\right)  $ is a unital
C*-algebra and is thus $\mathcal{A}^{+}$. Hence, for every $a\in
\mathcal{A}^{+},$ we have%
\[
\left\Vert a\right\Vert =\left\Vert \rho\left(  a\right)  \right\Vert \leq
\sup\left\{  \left\Vert \tau_{\alpha}\left(  a\right)  \right\Vert :\alpha\in
A\right\}  .
\]
Therefore the direct sum of the $\tau_{\alpha}$'s is faithful and $\left(
1\right)  $ is proved.

$\left(  3\right)  \Longrightarrow\left(  2\right)  $. This is obvious.

$\left(  4\right)  \Longrightarrow\left(  3\right)  $. It is clear that we
need only show that there is a faithful unital representation $\rho
:\mathcal{A}^{+}\rightarrow B\left(  H_{m}\right)  $. Suppose $\tau
:\mathcal{A}^{+}\rightarrow B\left(  M\right)  $ is an irreducible
representation, and suppose $D$ is a generating set with $Card\left(
D\right)  \leq m$. Let $\mathcal{A}_{0}$ be the unital $\ast$-subalgebra of
$\mathcal{A}^{+}$ over the field $\mathbb{Q}+i\mathbb{Q}$ of complex rational
numbers. Then $\mathcal{A}_{0}$ is norm dense in $\mathcal{A}$ and
$Card\mathcal{A}_{0}=CardD\leq m$. Suppose $f\in M$ is a unit vector. Since
$\tau$ is irreducible, $\tau\left(  \mathcal{A}_{0}\right)  f$ must be dense
in $M$. Suppose $B$ is an orthonormal basis for $M$, and, for each $e\in B$
let $U_{e}$ be the open ball centered at $e$ with radius $\sqrt{2}/2$. Each
$U_{e}$ must intersect the dense set $\tau\left(  \mathcal{A}_{0}\right)  f$,
and since the collection $\left\{  U_{e}:e\in B\right\}  $ is disjoint, we
conclude that
\[
\dim M=CardB\leq Card\tau\left(  \mathcal{A}_{0}\right)  f\leq Card\left(
\mathcal{A}_{0}\right)  \leq m.
\]
We know that for every $x\in\mathcal{A}_{0}$ there is an irreducible
representation $\tau_{x}:\mathcal{A}^{+}\rightarrow B\left(  M_{x}\right)  $
such that $\left\Vert \tau_{x}\left(  x\right)  \right\Vert =\left\Vert
x\right\Vert $. Since $\dim%
{\displaystyle\sum_{x\in\mathcal{A}_{0}}^{\oplus}}
M_{x}\leq m\cdot m=m$, there is a representation $\rho:\mathcal{A}%
^{+}\rightarrow B\left(  H_{m}\right)  $ that is unitarily to a direct sum of
$m$ copies of $%
{\displaystyle\sum_{x\in\mathcal{A}_{0}}^{\oplus}}
\tau_{x}$. Hence $\rho$ is isometric on the dense subset $\mathcal{A}_{0}$,
which implies $\rho$ is faithful.

$\left(  1\right)  \Longrightarrow\left(  3\right)  $. Since $\mathcal{A}$ is
$R_{<k}D$, we can choose a decomposition $H_{m}=%
{\displaystyle\sum_{\lambda\in\Lambda}^{\oplus}}
X_{\lambda}$ and representation $\pi=%
{\displaystyle\sum_{\lambda\in\Lambda}^{\oplus}}
\pi_{\lambda}$ as in part $\left(  1\right)  $ of Lemma \ref{decomp}. Now
$\left(  3\right)  $ follows from Lemma \ref{*SOTapprox}.
\end{proof}

\bigskip

We see that the class of $R_{<k}D$ algebras is closed under arbitrary free
products in the nonunital category of C*-algebras.

\bigskip

\begin{theorem}
\label{freeproduct}Suppose $k$ is an infinite cardinal and $\left\{
\mathcal{A}_{\iota}:i\in I\right\}  $ is a family of $R_{<k}D$ C*-algebras.
Then the free product $\underset{i\in I}{\ast}\mathcal{A}_{i}$ is $R_{<k}D$.
\end{theorem}

\begin{proof}
Choose an infinite cardinal $m\geq k$ $+$ $\underset{i\in I}{%
{\displaystyle\sum}
}Card\left(  \mathcal{A}_{i}\right)  $. Since $\underset{i\in I}{\ast
}\mathcal{A}_{i}$ is generated by $\mathcal{G}=\left[  \underset{i\in
I}{\bigcup}\mathcal{A}_{i}\right]  \backslash\left\{  0\right\}  \subseteq$
$\underset{i\in I}{\ast}\mathcal{A}_{i}$, clearly $\underset{i\in I}{\ast
}\mathcal{A}_{i}$ is $m$-generated. Choose a set $\Lambda$ with $Card\left(
\Lambda\right)  =m$ and let $S$ be the set of cardinals less than $k$. Write
\[
H_{m}=%
{\displaystyle\sum\nolimits_{s\in S}^{\oplus}}
{\displaystyle\sum\nolimits_{\lambda\in\Lambda}}
X_{s,\lambda}%
\]
where $\dim X_{s,\lambda}=s$ for every $s\in S$ and $\lambda\in\Lambda$. It
follows that, for each $i\in I$, we can find a representation $\pi
^{i}:\mathcal{A}_{i}\rightarrow B\left(  H_{m}\right)  $ such that%
\[
\pi^{i}=%
{\displaystyle\sum\nolimits_{s\in S}^{\oplus}}
{\displaystyle\sum\nolimits_{\lambda\in\Lambda}}
\pi_{s,\lambda}^{i}%
\]
satisfying $\left(  1\right)  $ and $\left(  2\right)  $ of Lemma
\ref{decomp}. Suppose $\varepsilon>0,$ $E\subseteq H_{m}$ is finite and
$W\subseteq\mathcal{G}$ is finite. We can write $W$ as a disjoint union of
$W_{i_{1}},\ldots,W_{i_{n}}$ with $W_{i}=W\cap\mathcal{A}_{i}$. Let $\rho_{i}$
be the restriction of $\rho$ to $\mathcal{A}_{i}$. Applying Lemma
\ref{*SOTapprox} to $\mathcal{A}_{i_{j}}$ and $\rho_{i_{j}}$ and $\pi^{i_{j}}$
for $1\leq j\leq n$, we can find one finite subset $G\subseteq S\times\Lambda$
so that if $P$ is the projection on $%
{\displaystyle\sum\nolimits_{\left(  s,\lambda\right)  \in G}^{\oplus}}
X_{s,\lambda}$, then there are unitary operators $U_{i_{1}},\ldots,U_{i_{n}%
}\in\mathcal{M}_{P}=PB\left(  H_{m}\right)  P$ so that, for $1\leq j\leq n,$
$a\in W_{j},$ $e\in E,$ we have%
\[
\left\Vert \left[  \rho_{i_{j}}\left(  a\right)  -U_{ij}^{\ast}\pi^{i_{j}%
}\left(  a\right)  U_{ij}\right]  e\right\Vert <\varepsilon\text{.}%
\]
Define $\tau_{ij}:\mathcal{A}_{i_{j}}^{+}\rightarrow\mathcal{M}_{P}$ by
\[
\tau_{ij}\left(  a\right)  =U_{ij}^{\ast}\pi^{i_{j}}\left(  a\right)  U_{ij},
\]
and for $i\in I\backslash\left\{  i_{1},\ldots,i_{n}\right\}  $ define
$\tau_{i}:\mathcal{A}_{i}\rightarrow\mathcal{M}_{P}$ by%
\[
\tau_{i}\left(  a\right)  =P\pi^{i}\left(  a\right)  P\text{.}%
\]
Then, by the definition of free product, there is a representation
$\tau:\underset{i\in I}{\ast}\mathcal{A}_{i}^{+}\rightarrow\mathcal{M}_{P}$
such that $\tau|\mathcal{A}_{i}=\tau_{i}$ for every $i\in I$. It follows that,
for every $e\in E$ and every $a\in W$,%
\[
\left\Vert \left[  \rho\left(  a\right)  -\tau\left(  a\right)  \right]
e\right\Vert <\varepsilon\text{.}%
\]
It follows from part $\left(  2\right)  $ of Lemma \ref{characterize} that
$\underset{i\in I}{\ast}\mathcal{A}_{i}$ is $R_{<k}D$.
\end{proof}

\bigskip

\begin{corollary}
Suppose $k$ is an infinite cardinal and $\left\{  \mathcal{A}_{\iota}:i\in
I\right\}  $ is a family of $R_{<k}D$ C*-algebras such that each
$\mathcal{A}_{i}$ has a one-dimensional unital representation. Then the unital
free product $\underset{i\in I}{\ast_{\mathbb{C}}}\mathcal{A}_{i}$ is
$R_{<k}D$.
\end{corollary}

\begin{proof}
This follows from the fact that if $\tau_{i}:\mathcal{A}_{i}\rightarrow
\mathbb{C}$ is a unital $\ast$-homomorphism for each $i\in I$, then
$\underset{i\in I}{\ast_{\mathbb{C}}}\mathcal{A}_{i}$ is $\ast$-isomorphic to
$\left(  \underset{i\in I}{\ast}\ker\tau_{i}\right)  ^{+}$.
\end{proof}

\bigskip

Without the condition on unital one-dimensional representations, the preceding
corollary is false. For example, $\underset{n\in\mathbb{N}}{\ast_{\mathbb{C}}%
}\mathcal{M}_{n}\left(  \mathbb{C}\right)  $ is not $RFD$ ( $=R_{<\aleph_{0}%
}D$ ), even though each $\mathcal{M}_{n}\left(  \mathbb{C}\right)  $ is $RFD$.
The reason is that each unital representation of the free product must be
injective on each $\mathcal{M}_{n}\left(  \mathbb{C}\right)  $ and must have
infinite-dimensional range. call an infinite cardinal $k$ a \emph{limit
cardinal}, if $k$ is the supremum of all the cardinals less than $k$.

However, there is something we can say about the general situation. If $k$ is
a limit cardinal, the \emph{cofinality }of $k$ is the smallest cardinal $s$
for which there is a set $E$ of cardinals less than $k$ whose supremum is $k.$
Clearly, the cofinality of $k$ is at most $k$. If $k$ is not a limit cardinal,
then there is a cardinal $s$ such that $k$ is the smallest cardinal larger
than $s,$ and if $E$ is a set of cardinals less than $k$, then $\sup\left(
E\right)  \leq s<k$.

\bigskip

\begin{theorem}
\label{unitafreeproduct}Suppose $k$ is an infinite cardinal and $\left\{
\mathcal{A}_{\iota}:i\in I\right\}  $ is a family of unital $R_{<k}D$
C*-algebras. Then

\begin{enumerate}
\item If $k$ is a limit cardinal and $Card\left(  I\right)  $ is less than the
cofinality of $k$, then the free product $\underset{i\in I}{\ast_{\mathbb{C}}%
}\mathcal{A}_{i}$ is $R_{<k}D$.

\item If $k$ is not a limit cardinal, then the free product $\underset{i\in
I}{\ast_{\mathbb{C}}}\mathcal{A}_{i}$ is $R_{<k}D$.
\end{enumerate}
\end{theorem}

\begin{proof}
$\left(  1\right)  .$ Choose $m\geq k+%
{\displaystyle\sum_{i\in I}}
Card\left(  \mathcal{A}_{i}\right)  $, and choose a set $\Lambda$ with
$Card\left(  \Lambda\right)  =m$. Using Lemma \ref{decomposition} we can, for
each $i\in I$, find a faithful representation $\pi^{i}=%
{\displaystyle\sum_{\lambda\in\Lambda}}
\pi_{\lambda,i}$ so that $\dim\pi^{i}=m$ and, for every $i\in I$ and
$\lambda\in\Lambda$, we have $\dim\pi_{\lambda,i}<k$. Since $Card\left(
I\right)  $ is less than the cofinality of $k$, we have, for each $\lambda
\in\Lambda$, a cardinal $s_{\lambda}<k$ such that $\sup_{i\in I}\dim
\pi_{\lambda,i}\leq s_{\lambda}$. If we replace each $\pi_{\lambda,i}$ with a
direct sum of $s_{\lambda}$ copies of itself, we get a new decomposition which
we will denote by the same names such that, for each $i$ and each $\lambda$ we
have $\dim\pi_{\lambda,i}=s_{\lambda}$. Hence we may write direct sum
decompositions of the $\pi^{i}$'s with respect to a common decomposition
$H_{m}=%
{\displaystyle\sum_{\lambda\in\Lambda}}
X_{\lambda}$ where $\dim X_{\lambda}=s_{\lambda}$ for every $\lambda\in
\Lambda$. The rest now follows as in the proof of Theorem \ref{freeproduct}.

$\left(  2\right)  $ If $k$ is not a limit cardinal, there is a largest
cardinal $s<k$. Repeat the proof of part $\left(  1\right)  $ with
$s_{\lambda}=s$ for every $\lambda\in\Lambda$.
\end{proof}

\bigskip

\begin{remark}
We cannot remove the condition on $Card\left(  I\right)  $ in part $\left(
1\right)  $ of Theorem \ref{unitalfreeproduct}. Suppose $k$ is a limit
cardinal and $I$ is a set of cardinals less than $k$ whose cardinality equals
the cofinality of $k$ and such that $\sup\left(  I\right)  =k$. For each
infinite cardinal $m,$ choose a set $\Lambda_{m}$ with cardinality $m$, and
let $\mathcal{S}_{m}$ denote the universal unital C*-algebra generated by
$\left\{  v_{\lambda}:\lambda\in\Lambda_{m}\right\}  $ with the conditions

\begin{enumerate}
\item $v_{\lambda}^{\ast}v_{\lambda}=1$ for every $\lambda\in\Lambda_{m},$

\item $v_{\lambda_{1}}v_{\lambda_{1}}^{\ast}v_{\lambda_{2}}v_{\lambda_{2}%
}^{\ast}=0$ for $\lambda_{1}\neq\lambda_{2}$ in $\Lambda_{m}$.\bigskip
\end{enumerate}

Since $\mathcal{S}_{m}$ is $m$-generated, it follows that every irreducible
representation of $\mathcal{S}_{m}$ is at most $m$-dimensional (see the proof
of $\left(  4\right)  \Longrightarrow\left(  3\right)  $ in Theorem
\ref{characterize}). Hence $\mathcal{S}_{m}$ is separated by $m$-dimensional
representations. On the other hand, if $\pi$ is a unital representation of
$\mathcal{S}_{m}$, then $\left\{  \pi\left(  v_{\lambda;}v_{\lambda}^{\ast
}\right)  :\lambda\in\Lambda_{m}\right\}  $ is an orthogonal family of nonzero
projections, which implies that the dimension of $\pi$ is at least $m.$It
follows that each $\mathcal{S}_{s}$ is $R_{<k}D$ for $s\in I.$ However, any
unital representation $\pi$ of the free product $\underset{s\in I}%
{\ast_{\mathbb{C}}}\mathcal{S}_{s}$ must induce a unital representation of
each $\mathcal{S}_{s}$, so its dimension is at least $\sup I=k.$ Hence
$\underset{s\in I}{\ast_{\mathbb{C}}}\mathcal{S}_{s}$ is not $R_{<k}D$.
\end{remark}

\section{Separable RFD Algebras}

In this section we show that for a separable C*-algebra being RFD is
equivalent to a lifting property.

Suppose $\left\{  e_{1},e_{2},\ldots\right\}  $ is an orthonormal basis for a
Hilbert space $\ell^{2}$, and, for each integer $n\geq1,$ let $P_{n}$ be the
projection onto $sp\left(  \left\{  e_{1},\ldots,e_{n}\right\}  \right)  $.
Let $\mathcal{M}_{n}=P_{n}B\left(  \ell^{2}\right)  P_{n}$ for $n\geq1$, and,
following Lemma \ref{simple}, let
\[
\mathcal{B=}\left\{  \left\{  T_{n}\right\}  \in%
{\displaystyle\prod_{n=1}^{\infty}}
\mathcal{M}_{n}:\exists T\in B\left(  \ell^{2}\right)  \text{ with }%
T_{n}\rightarrow T\text{ ( }\ast\text{-SOT )}\right\}  ,
\]
and let
\[
\mathcal{J=}\left\{  \left\{  T_{n}\right\}  \in\mathcal{B}:T_{n}%
\rightarrow0\text{ ( }\ast\text{-SOT )}\right\}  .
\]
Then, by Lemma \ref{simple}, we have that $\mathcal{B}$ is a unital
C*-algebra, $\mathcal{J}$ is a closed ideal in $\mathcal{B}$ and
\[
\pi\left(  \left\{  T_{n}\right\}  \right)  =\text{(}\ast\text{-SOT)-}%
\lim_{n\rightarrow\infty}T_{n}%
\]
defines a unital surjective $\ast$-homomorphism from $\mathcal{B}$ to
$B\left(  H\right)  $ whose kernel is $\mathcal{J}$. We can now give our
characterization of RFD for separable C*-algebras.

\begin{theorem}
\label{sep}Suppose $\mathcal{A}$ is a separable C*-algebra. The following are equivalent.

\begin{enumerate}
\item $\mathcal{A}$ is RFD

\item For every unital $\ast$-homomorphism $\rho:\mathcal{A}^{+}\rightarrow
B\left(  \ell^{2}\right)  $ there is a unital $\ast$-homomorphism
$\tau:\mathcal{A}^{+}\rightarrow\mathcal{B}$ such that $\pi\circ\tau=\rho$.
\end{enumerate}
\end{theorem}

\begin{proof}
The implication $\left(  2\right)  \Longrightarrow\left(  1\right)  $ is clear.

$\left(  1\right)  \Longrightarrow\left(  2\right)  $. Suppose $\mathcal{A=C}%
^{\ast}\left(  \left\{  a_{1},a_{2},\ldots\right\}  \right)  $ is $RFD$ and
$\rho:\mathcal{A}^{+}\rightarrow B\left(  \ell^{2}\right)  $ is a unital
$\ast$-homomorphism. It follows from Theorem \ref{characterize} that there is
an increasing sequence $\left\{  n_{k}\right\}  $ of positive integers and
unital $\ast$-homomorphisms $\tau_{k}:\mathcal{A\rightarrow M}_{n_{k}}$ such
that%
\[
\left\Vert \left[  \tau_{k}\left(  a_{j}\right)  -\rho\left(  a_{j}\right)
\right]  e_{i}\right\Vert <1/k
\]
for $1\leq i,j\leq k.$ It follows that $\tau_{n_{k}}\left(  a\right)
\rightarrow\rho\left(  a\right)  $ ( $\ast$-SOT ) for every $a\in
\mathcal{A}^{+}$. If $n_{k}<n<n_{k+1}$ we define $\tau_{n}:\mathcal{A}%
^{+}\mathcal{\rightarrow M}_{n}$ by%
\[
\tau_{n}\left(  a\right)  =\left(
\begin{array}
[c]{cccc}%
\tau_{n_{k}}\left(  a\right)  &  &  & \\
& \beta\left(  a\right)  &  & \\
&  & \ddots & \\
&  &  & \beta\left(  a\right)
\end{array}
\right)  ,
\]
where $\beta:\mathcal{A}^{+}\rightarrow\mathbb{C}$ is the unique $\ast
$-homomorphism with $\ker\beta=\mathcal{A}$, relative to the decomposition
\[
P_{n}\left(  \ell^{2}\right)  =P_{n_{k}}\left(  \ell^{2}\right)
\oplus\mathbb{C}e_{1+n_{k}}\oplus\cdots\oplus\mathbb{C}e_{-1+n_{k+1}}.
\]
It is easily seen that $\tau_{n}\left(  a\right)  \rightarrow\rho\left(
a\right)  $ ( $\ast$-SOT ) for every $a\in\mathcal{A}^{+}$. If we define
$\tau:\mathcal{A}\rightarrow\mathcal{B}$ by
\[
\tau\left(  a\right)  =\left\{  \tau_{n}\left(  a\right)  \right\}  ,
\]
we see that $\pi\circ\tau=\rho$.
\end{proof}

\begin{acknowledgement}
The author wishes to thank Tatiana Shulman and Terry Loring for bringing the
question answered by Theorem \ref{sep} to his attention. Thinking about this
question led to the discovery of all the results in this paper.
\end{acknowledgement}

\end{document}